\documentclass[10pt]{article}
\usepackage{amssymb, amsmath, amsthm, color,graphicx, amsfonts,a4}
\usepackage{amscd}

\definecolor{Red}{cmyk}{0,1,1,0}

\definecolor{verde}{cmyk}{1,0,1,0}

\definecolor{azul}{cmyk}{1,1,0,0}

\numberwithin{equation}{section}

\newcommand{\E}{\mathbb{E}}

\newcommand{\prob}{\mathbb{P}_{\textbf{p}}}

\newcommand{\be}{\begin{equation}}
\newcommand{\ee}{\end{equation}}

\newtheorem{teorema}{Theorem}
\newtheorem{proposicao}{Proposition}
\newtheorem{proposition}[teorema]{Proposition}
\newtheorem{definicao}{Definition}
\newtheorem{lema}{Lemma}

\begin{document}
\title{Truncated Connectivities in a highly supercritical anisotropic percolation model }
\author{Rodrigo G. Couto \footnote{ Departamento de
Matem{\'a}tica, Universidade Federal de Minas Gerais, Av.
Ant\^onio Carlos 6627 C.P. 702 CEP30123-970 Belo Horizonte-MG,
Brazil and Departamento de Matem\'atica, Universidade Federal de Ouro Preto, Rua Diogo de Vasconcelos 122 CEP35400-000 Ouro Preto-MG} ,
Bernardo N. B. de Lima \footnote{ Departamento de Matem{\'a}tica, Universidade Federal de Minas Gerais, Av. Ant\^onio
Carlos 6627 C.P. 702 CEP30123-970 Belo Horizonte-MG, Brazil} ,
R\'emy Sanchis$\;^{\dag}$}
\maketitle
\begin{abstract}
We consider an anisotropic bond percolation model on $\mathbb{Z}^2$,
with $\textbf{p}=(p_h,p_v)\in [0,1]^2$, $p_v>p_h$, and declare each
horizontal (respectively vertical) edge of $\mathbb{Z}^2$ to be open
with probability $p_h$(respectively $p_v$), and otherwise closed,
independently of all other edges. Let $x=(x_1,x_2) \in \mathbb{Z}^2$
with $0<x_1<x_2$, and $x'=(x_2,x_1)\in \mathbb{Z}^2$. It is natural
to ask how the two point connectivity function
$\prob(\{0\leftrightarrow x\})$ behaves, and whether anisotropy in
percolation probabilities implies the strict inequality
$\prob(\{0\leftrightarrow x\})>\prob(\{0\leftrightarrow x'\})$. In
this note we give an affirmative answer in the highly supercritical
regime.
\end{abstract}
{\footnotesize Keywords: percolation; anisotropy; truncated connectivity; supercritical phase }

\section{Introduction and Main Result}

Consider the square lattice $\mathbb{L}^2=(\mathbb{V},\mathbb{E})$, the graph where the set of vertices is $\mathbb{V}=\mathbb{Z}^2$ and the set of edges is  formed by pairs of nearest neighbors, i.e., $\mathbb{E}=\{e=\langle x,y\rangle\subset \mathbb{V}:d(x,y)=1\}$ where $d(x,y)$ is the usual path distance  in $\mathbb{Z}^2$. The set $\mathbb{E}$ is naturally partitioned in two disjoint sets, the set of horizontal edges, $\mathbb{E}^h=\{e=\langle x,y\rangle\in \mathbb{E}:x_2=y_2\}$ and $\mathbb{E}^v=\{e=\langle x,y\rangle\in \mathbb{E}:x_1=y_1\}$, the set of vertical edges.

Throughout this work, whenever $X$ is a set, we will denote its
cardinality by $|X|$. If $X \subset \mathbb{E}$ we will use
$X^h=X\cap \mathbb{E}^h$ to denote the set of horizontal edges of
$X$ and $X^v=X\cap \mathbb{E}^v$ to denote the set of vertical edges
of $X$.

We will now define the anisotropic bond percolation model. We define
a configuration $\omega$ of the process as a function
$\omega:\mathbb{E}\rightarrow \{0,1\}$. We call $\Omega$ the set of
configurations in $\mathbb{L}^2$. Given $\omega \in \Omega$ we say
that the edge $e$ is open if $\omega(e)=1$ and closed if
$\omega(e)=0$. Given $p_h, p_v \in [0,1]$ we state that each
horizontal edge $e \in \mathbb{E}^h$ is open with probability $p_h$
and closed with probability $1-p_h$ independently, and each vertical
edge $e \in \mathbb{E}^v$ is open with probability $p_v$ and closed
with probability $1-p_v$ independently. Denote by $\mu(e)$ the
Bernoulli measure with parameter $p_h$ if $e \in \mathbb{E}^h$ or
$p_v$ if $e \in \mathbb{E}^v$. To simplify the notation we use
$\textbf{p}=(p_h,p_v)$ to represent the parameters $p_h$ and $p_v$.
We assume throughout the paper that $p_h\leq p_v$.

Thus, the anisotropic bond percolation model on the square lattice is described by the probability space
$(\Omega, \mathcal{F},\prob)$ where $\Omega=\{0,1\}^{\E}$, $\mathcal{F}$ is the $\sigma$-algebra generated
by the cylinder sets in $\Omega$ and $\prob=\prod_{e\in\E}\mu(e)$ is the product of Bernoulli
measures $\mu(e)$.

Given a configuration $\omega \in \Omega$, an open cluster $A$ in $\omega$ is a connected subgraph
$A=(V_A,E_A)$ of $\mathbb{L}^2$ such that $\omega(e)=1,\forall$ $e \in E_A$ and $\omega(e)=0,\forall$ $e \in \partial_eA$, where $\partial_eA=\{e \in \mathbb{E}:|e \cap V_A|=1\}$ is the edge boundary
of $A$.
Given $x,y \in \mathbb{Z}^2$ denote by $\{x\leftrightarrow y\}$ the event where $x,y$ belongs the same
open cluster. Denote by $\tau_{\textbf{p}}(x,y)=\prob(\{x\leftrightarrow y\})$ the two-point
connectivity function.

In the isotropic case, that is $p_h=p_v$, given any pair of vertices
$x=(x_1,x_2), x'=(x_2,x_1) \in \mathbb{Z}^2_+$, with $x_1<x_2$, it
holds that  $\tau_{\textbf{p}}(0,x)=\tau_{\textbf{p}}(0,x')$, where
$0$ is the origin of $\mathbb{Z}^2$. Concerning the anisotropic
case, $p_h\neq p_v$, E. Andjel raised the question whether the
strict inequality $\tau_{\textbf{p}}(0,x)>\tau_{\textbf{p}}(0,x')$
holds or not. In spite of its apparent simplicity, this conjecture
(without restrictions) remains open. A more detailed discussion
about this conjecture and an affirmative answer when the parameters
$p_h$ and $p_v$ are sufficiently small can be found in \cite{LP}.

The result we establish is similar to the result of \cite{LP}, but
concerning the two-point truncated connectivity in the supercritical
regime. Given $x,y \in \mathbb{Z}^2$, we define
$\tau^f_{\textbf{p}}(x,y)$, the truncated connectivity function, as the probability that $x$ and $y$
belong
 the same finite open cluster, i.e.,
$$\tau^f_{\textbf{p}}(x,y)=\prob(\exists \mbox{ open cluster }A: \{x,y\} \subset V_A, |E_A|<\infty)$$

Given $p_h$ and $p_v$, define the constants $\lambda_h=\frac{1-p_h}{p_h}$, $\lambda_v=\frac{1-p_v}{p_v}$ and $\eta=\frac{\lambda_v}{\lambda_h}\in (0,1]$. Observe that for a fixed $p_h$, we have
$$\displaystyle\lim_{p_v \rightarrow {p_h}^+}\eta=1 \quad\mbox{ and }\quad\displaystyle\lim_{p_v \rightarrow 1}\eta=0.$$

Now we can state the main result of this paper:
\begin{teorema}\label{resultado2}

Given any $\rho >1$ and $\eta\in (0,1)$, there exists
$0<p^{\ast}(\eta,\rho)<1$ (close enough to $1$) such that
\begin{equation}\label{princ}\tau^{f}_{\textbf{p}}(0,x)>\tau^{f}_{\textbf{p}}(0,x')\end{equation}
holds for all $p_h>p^{\ast}(\eta,\rho)$ and
$x\in\mathbb{L}_\rho=:\{(y_1,y_2)\in\mathbb{Z}^2_+; y_2=\rho y_1\}$.
\end{teorema}

Observe that fixed $\eta$ and $p_h$ the Equation \ref{princ} is true
for any $p_v\geq\frac{1}{1+\eta\lambda_h}$. Fixed the parameters
$\textbf{p}=(p_h,p_v)$ the Equation \ref{princ} holds uniformly for
any $x\in\mathbb{L}_\rho$; observe that if $\rho\notin\mathbb{Q}$
then $\mathbb{L}_\rho=\emptyset$, otherwise $\mathbb{L}_\rho$ has
infinitely many points.

The paper is organized as follows. In Section 2, we give the necessary definitions for the rest of the paper. In Section 3, we give some lemmas about counting with contours. We will estimate the bounds for finite connectivity in Section 4. Finally, in Section 5, we complete the proof of Theorem \ref{resultado2}.

\section{Notations and Definitions}

In this section we will give a representation of the two-point truncated connectivity. Initially we need some notations and definitions, we follow \cite{ASGR} and \cite{APS}.

Given $V \subset \mathbb{Z}^2$ denote by $\partial_eV=\{e \in \mathbb{E}:|e\cap V|=1\}$ the edge boundary
of $V$. Denote by $\partial_v^{ext}V=\{x \in \mathbb{Z}^2\backslash V:d(x,V)=1\}$ the vertex external
boundary of $V$ and $\partial_v^{int}V=\{x \in V:d(x,\mathbb{Z}^2\backslash V)=1\}$ the vertex internal
boundary of $V$.

We say that a set $\gamma \subset \mathbb{E}$ is a cut set if the graph $\mathbb{L}^2\backslash \gamma \equiv (\mathbb{Z}^2,\mathbb{E} \backslash \gamma)$ is disconnected.

\begin{definicao}
A finite set $\gamma \subset \mathbb{E}$ is called a contour if
$\mathbb{L}^2\backslash \gamma$ has exactly one finite connected
component and $\gamma$ is minimal with respect to this property,
i.e. for all edges $e \in \gamma$ the graph
$(\mathbb{V},\mathbb{E}\backslash(\gamma \backslash e))$ has no
finite connected component. We denote by $\Gamma$ the set of all
contours in $\mathbb{L}^2$.
\end{definicao}

If $\gamma$ is a contour in $\mathbb{L}^2$, we denote by
$G_{\gamma}=(I_{\gamma},E_{\gamma})$ the unique finite connected
component of $\mathbb{L}^2\backslash \gamma$, where
$I_{\gamma}\subset \mathbb{Z}^2$ is the vertex interior of the
contour $\gamma$ and $E_{\gamma}\subset \mathbb{E}$ is the edge
interior of the contour $\gamma$. Observe that
$\partial_eI_{\gamma}=\gamma$.

Given a contour $\gamma \in \mathbb{L}^2$ and a set of vertices
$X \subset \mathbb{Z}^2$, we say that $\gamma$ surrounds $X$ if $X
\subset I_{\gamma}$. We denote by $\Gamma_X$ the set of contours that surround $X$ and by
$\Gamma_X^n\subset \Gamma_X$ the set of such contours
with cardinality $n$, i.e., $\Gamma_X^n=\{\gamma
\in \Gamma_X:|\gamma|=n\}$.

In fact, we will work in a finite subvolume of $\mathbb{L}^2$, consider
$G=(V_{G},E_{G})$ a finite subgraph of $\mathbb{L}^2$ and denote by
$\Omega_{G}$ the set of configurations in $G$, i.e.
$\Omega_G=\{\omega:E_G\rightarrow \{0,1\}\}$. Given $\omega \in
\Omega_G$ denote by $O(\omega)=\{e \in E_G: \omega(e)=1\}$ the set
of open edges in $\omega$ and $C(\omega)=\{e \in E_G: \omega(e)=0\}$
the set of closed edges in $\omega$.

Considering the measure $\prob$ restricted to $\Omega_G$ and fixed $\omega \in \Omega_G$, the probability
$\prob(\omega)$ is given explicitly by
\begin{eqnarray}
\prob(\omega)=p_h^{|O^h(\omega)|}p_v^{|O^v(\omega)|}(1-p_h)^{|C^h(\omega)|}(1-p_v)^{|C^v(\omega)|}
\end{eqnarray}

Given $V \subset \mathbb{Z}^2$, denote by
$\mathbb{L}^2[V]=(V,\mathbb{E}[V])$ the induced subgraph where
$\mathbb{E}[V]=\{\langle x,y\rangle \in \mathbb{E}:x \in V, y \in V\}$. We say
that a set $V \subset \mathbb{Z}^2$ is connected if
$\mathbb{L}^2[V]$ is a connected  graph. Denote by
$G_{N}=\mathbb{L}[V_N]$ the subgraph induced by
$V_N=[-N,N]\times[-N,N]\subset \mathbb{Z}^2$ and denote
$E_N=\mathbb{E}[V_N]$ to mean the edges in $G_N$.

Given any $x,y\in\mathbb{Z}^2$, we can choose some integer $N$ large
enough such that $\{x,y\}\subset V_N\backslash
\partial_v^{int}V_N$. The finite-volume finite connectivity
$\tau^{f,N}_{\textbf{p}}(x,y)$ is defined as the probability that
the vertices $x$ and $y$ belong the same open cluster and this
cluster does not intersect $\partial_v^{int}V_N$, that is
\begin{eqnarray}
\tau^{f,N}_{\textbf{p}}(x,y)=\prob(\exists \mbox{ open cluster }A:
\{x,y\} \subset V_A, V_A \cap
\partial_v^{int}V_N= \emptyset).
\end{eqnarray}

By continuity of the probability, we have that
\begin{eqnarray}
\tau^f_{\textbf{p}}(x,y)=\displaystyle\lim_{N \rightarrow \infty}\tau^{f,N}_{\textbf{p}}(x,y).
\end{eqnarray}
Once we obtain an upper bound for
$\tau^{f,N}_{\textbf{p}}(x,y)$ uniformly in $N$, the same bound also
holds for the limit $\tau^f_{\textbf{p}}(x,y)$.

Observing that

\begin{eqnarray}
&&1=\displaystyle\sum_{\omega \in \Omega_G}p_h^{|O^h(\omega)|}p_v^{|O^v(\omega)|}(1-p_h)^{|C^h(\omega)|}(1-p_v)^{|C^v(\omega)|}=
\nonumber \\
&&=\displaystyle\sum_{\omega \in \Omega_G}p_h^{|O^h(\omega)|+|C^h(\omega)|}p_v^{|O^v(\omega)|+|C^v(\omega)|}(\frac{1-p_h}{p_h})^{|C^h(\omega)|}(\frac{1-p_v}{p_v})^{|C^v(\omega)|}\nonumber \\
&&=p_h^{|E^{h}_G|}p_v^{|E^{v}_G|}\displaystyle\sum_{\omega \in \Omega_G}\lambda_h^{|C^h(\omega)|}\lambda_v^{|C^v(\omega)|}
\end{eqnarray}
we define the partition function $Z_{E_G}(\textbf{p})$ as
\begin{eqnarray}\label{f part}
&&Z_{E_G}(\textbf{p})=Z_{E_G}(p_h,p_v)=p_h^{-|E^{h}_G|}p_v^{-|E^{v}_G|}=\displaystyle\sum_{\omega
\in \Omega_G}\lambda_h^{|C^h(\omega)|}\lambda_v^{|C^v(\omega)|}.
\end{eqnarray}
When $G=G_N$, we write only $Z_N(\textbf{p})$. Analogously, the finite-volume
finite connectivity $\tau^{f,N}_{\textbf{p}}(x,y)$ can be written as

\begin{eqnarray}\label{exp con fin}
&&\tau^{f,N}_{\textbf{p}}(x,y)=\displaystyle\sum_{\substack{\omega \in \Omega_N: \exists \mbox{ open cluster }
A \\ \{x,y\}\subset V_A, \partial_eA \subset E_N}}p_h^{|O^h(\omega)|}p_v^{|O^v(\omega)|}(1-p_h)^{|C^h(\omega)|}(1-p_v)^{|C^v(\omega)|}= \nonumber \\
&&=\frac{1}{Z_N(\textbf{p})}\displaystyle\sum_{\substack{\omega \in \Omega_N: \exists \mbox{ open cluster }
A \\ \{x,y\}\subset V_A, \partial_eA \subset E_N}}\lambda_h^{|C^h(\omega)|}\lambda_v^{|C^v(\omega)|}
\end{eqnarray}
So, a configuration $\omega \in \Omega_N$ is given once we specify
the set of closed edges, $C(\omega)$, in $E_N$. Therefore, we can
rewrite (\ref{f part}) and (\ref{exp con fin}) as
\begin{eqnarray}
&&Z_N(\textbf{p})=\displaystyle\sum_{C \subset E_N}\lambda_h^{|C^h|}\lambda_v^{|C^v|}=p_h^{-|E^{h}_N|}p_v^{-|E^{v}_N|}
\end{eqnarray}
and
\begin{eqnarray}\label{carac}
&&\tau^{f,N}_{\textbf{p}}(x,y)=\frac{1}{Z_N(\textbf{p})}\displaystyle\sum_{\substack{C
\subset E_N \\ x\leftrightarrow y \mbox{ in }E_N\backslash C \\
\exists\gamma \in \Gamma_{\{x,y\}}, \gamma\subset C }}
\lambda_h^{|C^h|}\lambda_v^{|C^v|}
\end{eqnarray}

We will use the representation of the finite-volume finite connectivity given above to find
a lower bound and an upper bound to the finite connectivity.

\section{Contours}

We will give a characterization of the contours in terms of ${\mathbb{L}^2}^{\ast}=(\mathbb{V}^{\ast},\mathbb{E}^{\ast})$, the dual lattice of $\mathbb{L}^2$. The graph ${\mathbb{L}^2}^{\ast}$ is obtained translating $\mathbb{L}^2$ by the vector $\left( \frac{1}{2},\frac{1}{2} \right)$ (see \cite{G} for the formal definition). We have that ${\mathbb{L}^2}^{\ast}$
is isomorphic to $\mathbb{L}^2$ and we denote $0^{\ast}=(1/2,1/2)$ for the origin in ${\mathbb{L}^2}^{\ast}$. Observe that there is a bijection between the edges
of $\mathbb{L}^2$ and the edges of ${\mathbb{L}^2}^{\ast}$, since each edge $e$ of $\mathbb{L}^2$ is crossed by an unique edge
of the dual lattice, which we denote by $e^{\ast}$.

If $X \subset \mathbb{E}$, we let $X^{\ast}=\{e^{\ast}\in \mathbb{E}^{\ast}:e \in X\}$. Then we have the next proposition, whose proof can be found in \cite{ASGR}.

\begin{proposicao}
If $\gamma \in \Gamma$ then $\gamma^{\ast}$ is a circuit in ${\mathbb{L}^2}^{\ast}$.
\end{proposicao}

The next step is to estimate the number of contours surrounding $\{0,x\}$. Consider $x=(x_1,x_2)$ with $x_1,x_2>0$, given a contour $\gamma \in \Gamma_{\{0,x\}}$, we say that $\gamma$ is a minimal contour if
\begin{eqnarray}\label{tam cont}
|\gamma^h|=2(x_2+1)\quad and\quad |\gamma^v|=2(x_1+1)
\end{eqnarray}

We use the notation $\|x\|$ for the number of bonds of a minimal
contour that surrounds $\{0,x\}$, i.e., $\|x\|=2(x_1+x_2+2)$. Let
$\beta_x=|\Gamma_{\{0,x\}}^{\|x\|}|$  be the number of minimal
contours surrounding $\{0,x\}$.

We stress that the combinatorial estimative below is not optimal but
it is enough to proof Theorem \ref{resultado2}.

\begin{lema}\label{contagem}
 $|\Gamma_{\{0,x\}}^{\|x\|+m}|\leq 12^m \binom{\|x\|}{m} \beta_x$, $\mbox{ if }$ $0\leq m \leq \frac{\|x\|}{2}$
\end{lema}

\begin{proof}

Given $\gamma \in \Gamma^{\|x\|+m}_{\{0,x\}}$, the circuit $\gamma^{\ast}$ crosses the $x$-axis. Let then denote
by $e^{\ast}_{\gamma}$ the leftmost edge of $\gamma^{\ast}$ crossing the $x$-axis, that is, denote by $e_k^{\ast}=\langle v^{k}_{-},v^{k}_{+}\rangle$ the edge with $v^{k}_{-}=\left(k-\frac{1}{2},-\frac{1}{2}\right)$ and $v^{k}_{+}=\left(k-\frac{1}{2},\frac{1}{2}\right)$, by $k_{\gamma}=\min\{k\in\mathbb{Z}; \langle v^{k}_{-},v^{k}_{+}\rangle\in \gamma^{\ast}\}$ and by
$e^{\ast}_{\gamma}=\langle v^{k_{\gamma}}_{-},v^{k_{\gamma}}_{+}\rangle$ the desired edge. Since the circuit $\gamma^{\ast}$ surrounds $\{0,x\}$, we have $-m<k_{\gamma}\leq 0$.

Denote by $\Gamma^{\|x\|+m}_{\{0,x\}}(e_k^{\ast})=\{\gamma \in\Gamma^{\|x\|+m}_{\{0,x\}}:e^{\ast}_{\gamma}=
e_k^{\ast}\}$ the set of contours of cardinality $\|x\|+m$ surrounding $\{0,x\}$ whose leftmost edge
of the dual circuit $\gamma^{\ast}$ is $e_{k}^{\ast}$. Then we can write $\Gamma^{\|x\|+m}_{\{0,x\}}$ as the disjoint union
$$\Gamma^{\|x\|+m}_{\{0,x\}}=\displaystyle\bigcup_{k=-m+1}^{0}\Gamma^{\|x\|+m}_{\{0,x\}}(e_k^{\ast})$$

Observe that each contour $\gamma \in\Gamma^{\|x\|+m}_{\{0,x\}}(e_k^{\ast})$ can be constructed  from a minimal contour surrounding $\{0,x\}$ adding $m$ extra edges, each one has 3 possibilities to be chosen , as there are $\binom{\|x\|+m}{m}$ choices  for the positions of the added edges, it holds the following upper bound (with a lot of over-counting):

\begin{eqnarray}\label{assertion}
|\Gamma^{\|x\|+m}_{\{0,x\}}(e_k^{\ast})|\leq 3^m \cdot \binom{\|x\|+m}{m}\beta_{x} \quad\mbox{for}\quad k=-m+1,...,-1,0
\end{eqnarray}

A formal proof for the inequality above is given in the Appendix. With this upper bound, we obtain

\begin{eqnarray}\label{conta F 1}
&&|\Gamma^{\|x\|+m}_{\{0,x\}}|=\displaystyle\sum_{k=-m+1}^{0}|\Gamma^{\|x\|+m}_{\{0,x\}}(e_k^{\ast})|
\leq m \cdot 3^m \cdot \binom{\|x\|+m}{m}\beta_{x}=m \cdot 3^m
\binom{\|x\|}{m}\displaystyle\prod_{j=1}^{m}\frac{\|x\|+j}{\|x\|-m+j}\cdot\beta_{x} \nonumber \\
&& \leq m \cdot 3^m \cdot\binom{\|x\|}{m}\left(\frac{\|x\|+1}{\|x\|-m+1}\right)^m\cdot\beta_{x}\leq  12^m \cdot \binom{\|x\|}{m}\cdot\beta_{x}.
\end{eqnarray}

where in the last inequality we used that $m\leq 2^m$ and $\frac{\|x\|+1}{\|x\|-m+1}\leq 2$ for any positive integer $m\leq\frac{\|x\|}{2}$.
\end{proof}

The next lemma gives an estimate for the number of minimal contours that surround $\{0,x\}$.
\begin{lema}\label{lim beta}
For any positive integers $n$ and $\rho$, let us denote by $\alpha_n=\beta_{(n,\rho n)}$. Then, there exists the limit
$\displaystyle\lim_{n\rightarrow \infty}\alpha_n^{\frac{1}{n}}=\alpha$ and $1\leq \alpha \leq 4^{1+\rho}$.
\end{lema}
The proof follows by observing that $\alpha_{n+m}\geq \alpha_n \alpha_m$, and from a standard result on subadditive sequences (see
\cite{G}, page 399).

\section{Upper and Lower bounds on the finite connectivity}

In this section, we will give upper and lower bounds for the finite
connectivity, given by \eqref{carac}, when the parameters
$\lambda_h$ and $\lambda_v$ are sufficiently close to $0$.

\subsection{Upper bound}

Consider the representation of finite-volume finite connectivity
given by \eqref{carac},

Let us recall that $x=(x_1,x_2)$ and $x'=(x_2,x_1)$. Observe that
every set of closed edges $C \subset \mathbb{E}_N$ surrounding
$\{0,x'\}$ contains a contour $\gamma \in \Gamma$ surrounding
$\{0,x'\}$, then we can write

\begin{eqnarray*}
&&\tau_{\textbf{p}}^{f,N}(0,x')\leq
\frac{1}{Z_{N}(\bf{p})}\displaystyle\sum_{\gamma \in
\Gamma_{\{0,x'\}}}\sum_{\substack{C \subset E_{N} \\ C \supset
\gamma}}\lambda_h^{|C^h|}\lambda_v^{|C^v|}
\leq \frac{1}{Z_{N}(\bf{p})}\displaystyle\sum_{\gamma \in \Gamma_{\{0,x'\}}}\left(\lambda_h^{|\gamma^h|}\lambda_v^{|\gamma^v|}\sum_{\substack{C \subset E_{N}\backslash \gamma \\
\exists\delta \in \Gamma_{\{0,x'\}}, \delta\subset (C \cup \gamma )}}\lambda_h^{|C^h|}\lambda_v^{|C^v|}\right)\\
&&=\displaystyle\sum_{\gamma \in
\Gamma_{\{0,x'\}}}\lambda_h^{|\gamma^h|}\lambda_v^{|\gamma^v|}\left(\frac{1}{Z_{N}(\textbf{p})}\sum_{\substack{C
\subset E_{N}\backslash \gamma \\ \exists\delta \in
\Gamma_{\{0,x'\}}, \delta\subset (C \cup \gamma
)}}\lambda_h^{|C^h|}\lambda_v^{|C^v|}\right)\leq
\displaystyle\sum_{\gamma \in \Gamma_{\{0,x'\}}}\lambda_h^{|\gamma^h|}\lambda_v^{|\gamma^v|}\\
&&=\displaystyle\sum_{n \geq \|x\|}\sum_{\gamma \in
\Gamma_{\{0,x'\}}^n}\lambda_h^{|\gamma^h|}\lambda_v^{|\gamma^v|},\nonumber
\end{eqnarray*}
where $\|x\|=2(x_1+x_2+2)$. Using \eqref{tam
cont}, we obtain

\begin{eqnarray*}
&&\tau_{\textbf{p}}^{f,N}(0,x')\leq \lambda_h^{2(x_1+1)}\lambda_v^{2(x_2+1)}\displaystyle\sum_{n \geq \|x\|}\sum_{\gamma \in \Gamma_{\{0,x'\}}^n}\lambda_h^{|\gamma^h|-2(x_1+1)}\lambda_v^{|\gamma^v|-2(x_2+1)}\leq \\
&&\leq \lambda_h^{2(x_1+1)}\lambda_v^{2(x_2+1)}\displaystyle\sum_{n \geq \|x\|}\sum_{\gamma \in \Gamma_{\{0,x'\}}^n}\lambda_h^{|\gamma|-2(|x|+2)}=\\
&&=\lambda_h^{2(x_1+1)}\lambda_v^{2(x_2+1)}\displaystyle\sum_{n \geq
\|x\|}\left( \lambda_h^{n-\|x\|} \sum_{\gamma \in
\Gamma_{\{0,x'\}}^n}1\right),
\end{eqnarray*}
where we used $\lambda_v
\leq \lambda_h$. Then, observing that
$|\Gamma^{n}_{\{0,x\}}|=|\Gamma^{n}_{\{0,x'\}}|$,
\begin{eqnarray}\label{serie}
\tau_{\textbf{p}}^{f,N}(0,x')\leq
\lambda_h^{2(x_1+1)}\lambda_v^{2(x_2+1)}\displaystyle\sum_{n
=\|x\|}^{\frac{3\|x\|}{2}} \lambda_h^{n-\|x\|}|\Gamma^{n}_{\{0,x\}}|
+ \lambda_h^{2(x_1+1)}\lambda_v^{2(x_2+1)}\displaystyle\sum_{n >
\frac{3\|x\|}{2}} \lambda_h^{n-\|x\|}|\Gamma^{n}_{\{0,x\}}|.
\end{eqnarray}

To bound the second term in the right hand side of $\eqref{serie}$, we
use the trivial bound  $|\Gamma^{n}_{\{0,x\}}|\leq 4^n$:
\begin{eqnarray}\label{ngrande}
&&\lambda_h^{2(x_1+1)}\lambda_v^{2(x_2+1)}\displaystyle\sum_{n >
\frac{3\|x\|}{2}}
\lambda_h^{n-\|x\|}|\Gamma^{n}_{\{0,x\}}|\leq
\lambda_h^{2(x_1+1)}\lambda_v^{2(x_2+1)}\displaystyle\sum_{n >
\frac{3\|x\|}{2}} \lambda_h^{n-\|x\|}4^n
< \nonumber \\
&&< \lambda_h^{2(x_1+1)}\lambda_v^{2(x_2+1)}\displaystyle\sum_{j > \frac{\|x\|}{2}} \lambda_h^{j}(4^3)^j.
\end{eqnarray}

If $\lambda_h < 4^{-3}$, then
\begin{eqnarray}
&&\lambda_h^{2(x_1+1)}\lambda_v^{2(x_2+1)}\displaystyle\sum_{n >
\frac{3\|x\|}{2}} \lambda_h^{n-\|x\|}|\Gamma^{n}_{\{0,x\}}| <
\lambda_h^{2(x_1+1)}\lambda_v^{2(x_2+1)}\frac{(4^3
\lambda_h)^{\frac{\|x\|}{2}+1}}{1-4^3\lambda_h}.
\end{eqnarray}

Concerning now the first term in the right hand side of
$\eqref{serie}$, we use Lemma \ref{contagem}, then

$$|\Gamma^{\|x\|+m}_{\{0,x\}}|\leq 12^{m}\beta_x \binom{\|x\|}{m} \mbox{ if }m \leq \frac{\|x\|}{2}.$$

Then
\begin{eqnarray}\label{npequeno}
&&\lambda_h^{2(x_1+1)}\lambda_v^{2(x_2+1)}\displaystyle\sum_{n
=\|x\|}^{\frac{3\|x\|}{2}}
\lambda_h^{n-\|x\|}|\Gamma^{n}_{\{0,x\}}|=
\lambda_h^{2(x_1+1)}\lambda_v^{2(x_2+1)}\displaystyle\sum_{m=0}^{\frac{\|x\|}{2}} \lambda_h^{m}|\Gamma^{\|x\|+m}_{\{0,x\}}|\leq \nonumber \\
&&\leq
\lambda_h^{2(x_1+1)}\lambda_v^{2(x_2+1)}\beta_x\displaystyle\sum_{m=0}^{\frac{\|x\|}{2}}
\lambda_h^{m}12^m \binom{\|x\|}{m}<
\lambda_h^{2(x_1+1)}\lambda_v^{2(x_2+1)}\beta_x\displaystyle\sum_{m=0}^{\|x\|}
(12\lambda_h)^{m}
\binom{\|x\|}{m}= \nonumber \\
&&=\lambda_h^{2(x_1+1)}\lambda_v^{2(x_2+1)}\beta_x(1+12\lambda_h)^{\|x\|}.
\end{eqnarray}

Thus, from $\eqref{ngrande}$ and $\eqref{npequeno}$ we get
\begin{eqnarray}
&&\tau_{\textbf{p}}^{f,N}(0,x')<\lambda_h^{2(x_1+1)}\lambda_v^{2(x_2+1)}\left(\frac{(4^3
\lambda_h)^{\frac{\|x\|}{2}+1}}{1-4^3\lambda_h}
+\beta_x(1+12\lambda_h)^{\|x\|} \right).
\end{eqnarray}
Observing that the bound above does not depend on $N$, we can
take $N\rightarrow\infty$ and get
\begin{eqnarray}\label{cs}
&&\tau_{\textbf{p}}^{f}(0,x')<\lambda_h^{2(x_1+1)}\lambda_v^{2(x_2+1)}\left(\frac{(4^3
\lambda_h)^{\frac{\|x\|}{2}+1}}{1-4^3\lambda_h}
+\beta_x(1+12\lambda_h)^{\|x\|} \right)
\end{eqnarray}

\subsection{Lower bound}

Given $\gamma \in \Gamma_{\{0,x\}}^{\|x\|}$, a minimal contour surrounding $\{0,x\}$, with $\gamma \subset E_N$ (we can suppose $N$ large enough such that $E_N$ contains $\gamma$), we have that $\gamma$ contains all edges in the set
$$Q=\{\langle(-1,0),(0,0)\rangle, \langle(0,-1),(0,0)\rangle, \langle(x_1,x_2),(x_1,x_2+1)\rangle, \langle(x_1,x_2),(x_1+1,x_2)\rangle\}$$ and
$(\gamma \backslash Q)^{\ast}$ is the union of two disjoint paths in
the dual lattice, $c_1^{\ast}(\gamma)$, connecting $(-1/2,1/2)$ to
$(x_1-1/2,x_2+1/2)$ and $c_2^{\ast}(\gamma)$, connecting
$(1/2,-1/2)$ to $(x_1+1/2,x_2-1/2)$, both paths with cardinality
$x_1+x_2$. Denote by $\sigma_1(\gamma)$ the path in the original
lattice of minimal length connecting $0$ to $x$ obtained by
translating the path $c_1^{\ast}(\gamma)$ by $(1/2,-1/2)$, and
$\sigma_2(\gamma)$ the path of minimum size connecting $0$ to $x$
obtained by translating the path $c_2^{\ast}(\gamma)$ by
$(-1/2,1/2)$. Consider the set $t_{\gamma}=\sigma_1(\gamma) \cup
\sigma_2(\gamma)$. Note that
\begin{eqnarray}\label{arv cont}
&&|t_{\gamma}^h|\leq 2x_1< |\gamma^v| \quad \text{and} \quad |t_{\gamma}^v|\leq 2x_2< |\gamma^h|.
\end{eqnarray}

By the definition of contour, the set $E_N \backslash \gamma$ is partitioned into two disjoint subsets, $E_{\gamma}$
being the edge interior of the contour $\gamma$ and $E_N \backslash (\gamma \cup E_{\gamma})$ called the set of
exterior edges of the contour $\gamma$, where $G_{\gamma}=(I_{\gamma},E_{\gamma})$ is connected.

Denote by $C_{\gamma}$ the set of configurations of closed edges in $E_N$ such that $C_{\gamma}=\{C\subset E_N; C \supset \gamma\mbox{ and }C \cap t_{\gamma}=\emptyset$ (that is, the edges of $t_{\gamma}$ are open), clearly $C_{\gamma}\neq \emptyset$.
Observe that $C_{\tilde{\gamma}}\cap C_{\gamma}=\emptyset$ for all $\tilde{\gamma},\gamma\in \Gamma_{\{0,x\}}^{\|x\|}, \tilde{\gamma}\neq\gamma$. Because, if $\tilde{\gamma}\neq \gamma$, then either $t_{\tilde{\gamma}}\cap \gamma \neq \emptyset$, or
$t_{\gamma}\cap \tilde{\gamma} \neq \emptyset$. As in a configuration of $C_{\gamma}$ the edges in $\gamma$
are closed and the edges in $t_{\gamma}$ are open, we have $C_{\tilde{\gamma}}\cap C_{\gamma}=\emptyset$.

Consider the disjoint union $\mathcal{C}=\displaystyle\bigcup_{\gamma \in \Gamma_{\{0,x\}}^{\|x\|}}\mathcal{C}_{\gamma}$ where each $\mathcal{C}_{\gamma}\neq \emptyset$.

We have

\begin{eqnarray}\label{inf0}
&&\tau_{\textbf{p}}^{f,N}(0,x) \geq \prob(\mathcal{C})=\prob\left(\displaystyle\bigcup_{\gamma \in \Gamma_{\{0,x\}}^{\|x\|}}\mathcal{C}_{\gamma}\right)=\displaystyle\sum_{\gamma \in \Gamma_{\{0,x\}}^{\|x\|}}
\prob(\mathcal{C}_{\gamma})= \nonumber \\
&&=\beta_x(1-p_h)^{|\gamma^h|}(1-p_v)^{|\gamma^v|}p_h^{|t_{\gamma}^h|}p_v^{|t_{\gamma}^v|}=\beta_x\frac{(1-p_h)^{|\gamma^h|}}{p_h^{|\gamma^h|}}
\frac{(1-p_v)^{|\gamma^v|}}{p_v^{|\gamma^v|}}p_h^{|\gamma^h|+|t_{\gamma}^h|}p_v^{|\gamma^v|+|t_{\gamma}^v|}>\nonumber \\
&&>\beta_x\lambda_h^{|\gamma^h|}\lambda_v^{|\gamma^v|}p_h^{|\gamma^h|+|\gamma^v|}p_v^{|\gamma^h|+|\gamma^v|}=
\beta_x\lambda_h^{|\gamma^h|}\lambda_v^{|\gamma^v|}p_h^{|\gamma|}p_v^{|\gamma|}>\beta_x\lambda_h^{|\gamma^h|}\lambda_v^{|\gamma^v|}p_h^{2|\gamma|}= \nonumber \\
&&=\beta_x\lambda_h^{2(x_2+1)}\lambda_v^{2(x_1+1)}p_h^{2\|x\|}.
\end{eqnarray}
In the first inequality we used \eqref{arv cont}, in the second inequality we used that $p_v>p_h$, and in the last equality we used \eqref{tam cont}.

Note that the above bound is uniform in $N$, so we can write
\begin{eqnarray}\label{ci}
&&\tau_{\textbf{p}}^{f}(0,x) \geq \beta_x\lambda_h^{2(x_2+1)}\lambda_v^{2(x_1+1)}p_h^{2\|x\|}.
\end{eqnarray}

We can summarize the results of this section in the following
proposition:

\begin{proposition}Let $x=(x_1,x_2)$ and  $x'=(x_2,x_1)\in\mathbb{Z}^2$. Then
\begin{equation}\label{lower}
\tau_{\textbf{p}}^{f}(0,x) \geq \beta_x\lambda_h^{2(x_2+1)}\lambda_v^{2(x_1+1)}p_h^{2\|x\|}
\end{equation} and
\begin{equation}\label{upper}\tau_{\textbf{p}}^{f}(0,x')<\lambda_h^{2(x_1+1)}\lambda_v^{2(x_2+1)}\left(\frac{(4^3
\lambda_h)^{\frac{\|x\|}{2}+1}}{1-4^3\lambda_h}
+\beta_x(1+12\lambda_h)^{\|x\|} \right).
\end{equation}
\end{proposition}

\section{Proof of Theorem \ref{resultado2}}

Consider $x=(x_1,x_2) \in \mathbb{Z}_{+}^2$, $x'=(x_2,x_1)$,  $\rho=\frac{x_2}{x_1}>1$ and $\eta=\frac{\lambda_v}{\lambda_h}$. Comparing \eqref{lower} and \eqref{upper} we see that $\tau^{f}_{\textbf{p}}(0,x)>\tau^{f}_{\textbf{p}}(0,x')$ if

$$\lambda_h^{2(x_1+1)}\lambda_v^{2(x_2+1)}\left(\frac{(4^3 \lambda_h)^{\frac{\|x\|}{2}+1}}{1-4^3\lambda_h} +\beta_x(1+12\lambda_h)^{\|x\|} \right)<\beta_x\lambda_h^{2(x_2+1)}\lambda_v^{2(x_1+1)}p_h^{2\|x\|},$$
which is equivalent to
$$\eta<\left(\frac{\beta_x p_h^{2\|x\|}}{\frac{(4^3 \lambda_h)^{\frac{\|x\|}{2}+1}}{1-4^3\lambda_h} +\beta_x(1+12\lambda_h)^{\|x\|}}\right)^{\frac{1}{2(x_2-x_1)}}.$$

Then, $\tau^{f}_{\textbf{p}}(0,x)>\tau^{f}_{\textbf{p}}(0,x')$ holds for all
$\eta$ which satisfy $$\eta < \tilde{\eta}(p_h,x_1,x_2),$$
where $\tilde{\eta}(p_h,x_1,x_2)$ is defined for all $\lambda_h<4^{-3}$, i.e. $p_h>(1+4^{-3})^{-1}$,
and for all $x_2>x_1>0$ and it is given explicitly by
\begin{eqnarray}\label{f eta}
&&\tilde{\eta}(p_h,x_1,x_2)=\left(\frac{\beta_x p_h^{4(x_1+x_2+2)}}{\frac{(4^3 \lambda_h)^{x_1+x_2+3}}{1-4^3\lambda_h} +\beta_x(1+12\lambda_h)^{2(x_1+x_2+2)}}\right)^{\frac{1}{2(x_2-x_1)}}.
\end{eqnarray}
Note now that
\begin{eqnarray}
&&\displaystyle\lim_{p_h\rightarrow 1^{-}}\tilde{\eta}(p_h,x_1,x_2)=\displaystyle\lim_{p_h\rightarrow 1^{-}}
\left(\frac{\beta_x^{\frac{1}{x_1}} p_h^{4(1+\rho+\frac{2}{x_1})}}{\frac{(4^3 \lambda_h)^{1+\rho+\frac{3}{x_1}}}{(1-4^3\lambda_h)^{\frac{1}{x_1}}} +\beta_x^{\frac{1}{x_1}}(1+12\lambda_h)^{2(1+\rho+\frac{2}{x_1})}}\right)^{\frac{1}{2(\rho-1)}}=1 \nonumber
\end{eqnarray}
since $\lambda_h\rightarrow 0^{+}$ when $p_h\rightarrow 1^{-}$.

Note also that, for a fixed slope $\rho>1$ and using Lemma $\ref{lim beta}$ we have that
\begin{eqnarray}
&&\displaystyle\lim_{|x|\rightarrow \infty}\tilde{\eta}(p_h,x_1,x_2)=\displaystyle\lim_{x_1\rightarrow \infty}
\left(\frac{\beta_x^{\frac{1}{x_1}} p_h^{4(1+\rho+\frac{2}{x_1})}}{\frac{(4^3 \lambda_h)^{1+\rho+\frac{3}{x_1}}}{(1-4^3\lambda_h)^{\frac{1}{x_1}}} +\beta_x^{\frac{1}{x_1}}
(1+12\lambda_h)^{2(1+\rho+\frac{2}{x_1})}}\right)^{\frac{1}{2(\rho-1)}}=f(p_h,\rho), \nonumber
\end{eqnarray}
where $$f(p_h,\rho)=\left(\frac{\alpha p_h^{4(1+\rho)}}{(4^3
\lambda_h)^{1+\rho}
+\alpha(1+12\lambda_h)^{2(1+\rho)}}\right)^{\frac{1}{2(\rho-1)}},$$
and $\displaystyle\lim_{p_h\rightarrow 1^{-}}f(p_h,\rho)=1$.

This concludes the proof.

\section {Appendix: Proof of the Inequality \ref{assertion}}

\begin{proposition}For any $k=-m+1,...,-1,0$
\begin{eqnarray}
|\Gamma^{\|x\|+m}_{\{0,x\}}(e_k^{\ast})|\leq 3^m \cdot \binom{\|x\|+m}{m}\beta_{x}.
\end{eqnarray}
\end{proposition}
\begin{proof}

We adapt for contours the argument of Theorem V.6.9 of \cite{S} used originally for paths. Given $\gamma \in \Gamma^{\|x\|}_{\{0,x\}}$, $\gamma^{\ast}$ is a circuit in the dual $\mathbb{L}^{\ast}$
with cardinality $\|x\|$ containing  the vertices $u=(-\frac{1}{2},-\frac{1}{2})$ and
$w=(x_1+\frac{1}{2},x_2+\frac{1}{2})$. Therefore $\gamma^{\ast}$ is the concatenation of two paths on the dual lattice,
$c_1^{\ast}$ and $c_2^{\ast}$, both of cardinality $\frac{\|x\|}{2}=x_1+x_2+2$, that do not intersect, $c_1^{\ast}$ being an oriented path from $u$ to $w$, and $c_2^{\ast}$ an oriented path from $w$ to $u$.

The path $c_1^{\ast}$ has exactly $x_1+1$ steps to the right and $x_2+1$ steps upwards. We can then identify $c_1^{\ast}$ as a word $W(c_1^{\ast})$ with $x_1+x_2+2$ letters, containing $x_1+1$ letters
$R$(right) and $x_2+1$ letters $U$(up).
Analogously we can identify $c_2^{\ast}$ as a word $W(c_2^{\ast})$ with $x_1+x_2+2$ letters, containing $x_1+1$ letters $L$(left) and $x_2+1$ letters $D$(down). Thereby, we identify the circuit $\gamma^{\ast}$ as a word $W(\gamma^{\ast})$ with $\|x\|$ letters, the concatenation of the words $W(c_1^{\ast})$ and $W(c_2^{\ast})$.

Analogously, any circuit $\tilde{\gamma} \in \Gamma^{\|x\|+m}_{\{0,x\}}(e_k^{\ast})$, can be identified with a word $W(\tilde{\gamma}^{\ast})$ of $\|x\|+m$ entries, using the alphabet  $\{R,L,U,D\}$. This word is obtained as follows: writing
$\tilde{\gamma}^{\ast}=u_1,f_1^{\ast},u_2,...,f_{\|x\|+m}^{\ast},u_{\|x\|+m+1}$ where $u_{\|x\|+m+1}=u_1=v_{-}^k$
and $f_1^{\ast}=e_k^{\ast}$, we identify $f_i^{\ast}=\{u_i,u_{i+1}\}$ with the letter $U,D,R$ or $L$ depending on $u_{i+1}-u_i$ being equal to $(0,1),\ (0,-1), (1,0)$ or $(-1,0)$, respectively.

We will show that the word $W(\tilde{\gamma}^{\ast})$ contains a
subword $W$ such that the dual circuit $\gamma^{\ast}$, defined by
the word $W$ from the vertex $u$, is the dual of a minimal contour
$\gamma$. As $\tilde{\gamma}\in\Gamma_{\{0,x\}}$, the dual contour
$\tilde{\gamma}^{\ast}$ has at least one vertex in the region
$\{(a,b)\in \mathbb{R}^2:a> x_1, b> x_2\}$. Then we can choose one
vertex $u_{i_0+1}$ with the property  $d(u_{i_0+1},x) \geq
d(z,x),\forall z\in \{u_1,u_2,...,u_{\|x\|+m+1}\}\cap\{(a,b)\in
\mathbb{R}^2:a> x_1, b> x_2\}$. Observe that the letters associated
to the edges $f_{i_0}^{\ast}=\{u_{i_0},u_{i_0+1}\}$ and
$f_{i_0+1}^{\ast}=\{u_{i_0+1},u_{i_0+2}\}$ are $R$ and $D$,
respectively.

Define the paths
$\lceil\tilde{\gamma}^{\ast}\rceil=u_1,f_1^{\ast},...,f_{i_0}^{\ast},u_{i_0+1}$
and
$\lfloor\tilde{\gamma}^{\ast}\rfloor=u_{i_0+1},f_{i_0+1}^{\ast},...,f_{\|x\|+m}^{\ast},u_{\|x\|+m+1}$.
By construction, the word $W(\lceil\tilde{\gamma}^{\ast}\rceil)$
contains at least $x_1+1$ letters $R$ and at least $x_2+1$ letters
$U$, and the word $\lfloor\tilde{\gamma}^{\ast}\rfloor$ contains at
least $x_1+1$ letters $L$ and at least $x_2+1$ letters $D$. Let
$\lceil W\rceil$ be the subword obtained from
$W(\lceil\tilde{\gamma}^{\ast}\rceil)$ by taking in
$W(\lceil\tilde{\gamma}^{\ast}\rceil)$ the first $x_2+1$ letters $U$
and the last $x_1+1$ letters $R$. Analogously let $\lfloor W\rfloor$
the subword obtained from $W(\lfloor\tilde{\gamma}^{\ast}\rfloor)$
by taking in $W(\lfloor\tilde{\gamma}^{\ast}\rfloor)$ the first
$x_2+1$ letters $D$ and the last $x_1+1$ letters $L$. Denote by $W$
the concatenation of words $\lceil W\rceil$ and $\lfloor W\rfloor$
and by $\gamma^{\ast}$ the realizations of the word $W$ from $u$.
Note that $|\gamma^{\ast}|=\|x\|$. We have that
$\gamma^{\ast}$ is the concatenation of paths $\lceil
\gamma^{\ast}\rceil$ and $\lfloor \gamma^{\ast}\rfloor$ where
$\lceil \gamma^{\ast}\rceil$ is the realization of the subword
$\lceil W\rceil$ from $u$ and $\lfloor \gamma^{\ast}\rfloor$ is the
realization of the subword $\lfloor W\rfloor$ from $w$. Consider
$\lfloor \gamma^{\ast}\rfloor^{-1}$ the inverse path of $\lfloor
\gamma^{\ast}\rfloor$ leaving $u$ and arriving $w$. The word
$W(\lfloor \gamma^{\ast}\rfloor^{-1})$ is obtained from $\lfloor W
\rfloor$ starting from the last letter of $\lfloor W \rfloor$ to the
first, and replacing $D$ by $U$ and $L$ by $R$. Clearly the first
letter of $\lceil W \rceil$ is $U$, the last letter of $\lceil W
\rceil$ is $R$ and the last letter of $W(\lfloor
\gamma^{\ast}\rfloor^{-1})$ is $U$. By construction of
$\gamma^{\ast}$, the first letter of $W(\lfloor
\gamma^{\ast}\rfloor^{-1})$ is $R$.
%
%

To show that $\gamma^{\ast}$ is a circuit in the dual having
$\{0,x\}$ in its interior, it remains to be seen that the paths
$\lceil \gamma^{\ast}\rceil$ and $\lfloor \gamma^{\ast}\rfloor^{-1}$
connecting $u$ to $w$ have no intersection, except for $u$ and $w$.
Then, suppose that there exists a vertex
$v=\left(r-\frac{1}{2},s-\frac{1}{2}\right)$ distinct from $u$ and
$w$ at the intersection of $\lceil \gamma^{\ast}\rceil$ and $\lfloor
\gamma^{\ast}\rfloor^{-1}$. We write $\lceil
\gamma^{\ast}\rceil(u,v)$ to denote the subpath of $\lceil
\gamma^{\ast}\rceil$ from $u$ to $v$, analogously for $\lfloor
\gamma^{\ast}\rfloor^{-1}(u,v)$. Thereby, the words $W(\lceil
\gamma^{\ast}\rceil(u,v))$ and $W(\lfloor
\gamma^{\ast}\rfloor^{-1}(u,v))$ contains $r$ letters $R$ and $s$
letters $U$. Consider the subpath of $\lceil
\tilde{\gamma}^{\ast}\rceil$ from $u_1$ to $u_{i_{1}}$, $\lceil
\tilde{\gamma}^{\ast}\rceil(u_1,u_{i_1})$, where $$i_1=\max\bigl\{i:
W(\lceil \tilde{\gamma}^{\ast}\rceil(u_1,u_i)) \mbox{ has exactly }
s \mbox{ letters } U \mbox{ and } W(\lceil
\tilde{\gamma}^{\ast}\rceil(u_i,u_{i_0+1})) \mbox{ has exactly }
x_1+1-r \mbox{ letters } R \bigr\}.$$

Observe that $i_1<i_0+1$ and in the construction of the path $\lceil
\gamma^{\ast}\rceil$ from the path $\lceil
\tilde{\gamma}^{\ast}\rceil$, the subpath $\lceil
\gamma^{\ast}\rceil(u,v)$ is obtained when we restrict the
construction to the subpath $\lceil
\tilde{\gamma}^{\ast}\rceil(u_1,u_{i_1})$.

With respect to the final vertex $u_{i_1}$ of the path $\lceil
\tilde{\gamma}^{\ast}\rceil(u_1,u_{i_1})$, we have that the second
coordinate of $u_{i_1}$ is at most $s$. Moreover, the first
coordinate must be at least $r$, because otherwise the path $\lceil
\tilde{\gamma}^{\ast}\rceil(u_{i_1},u_{i_0})$ would have more than $x_1+1-r$
edges associated to the letter $R$. But note that from the stage of
construction of $\lceil \gamma^{\ast}\rceil$ where a symbol $R$
appears in the word $\lceil W \rceil$, extracted from $W(\lceil
\tilde{\gamma}^{\ast}\rceil)$, the next letters $R$ of the word
$W(\lceil \tilde{\gamma}^{\ast}\rceil)$, from that first letter $R$
extracted, are all on the word $\lceil W \rceil$. Therefore, the
word $\lceil W \rceil$ would contain more than $x_1+1-r+r=x_1+1$
letters $R$, which is a contradiction. Then, the first coordinate of
$u_{i_1}$ must be at least $r$.

We conclude that the path $\lceil \tilde{\gamma}^{\ast}\rceil(u_1,u_{i_1})$ leaves the vertex $u_1$ through
the edge $f_1^{\ast}=e_k^{\ast}$, crossing the half-line $\{(a,0)\in \mathbb{R}^2:a\leq k\}$ of the original lattice, thereafter no longer intersects this half-line, always remaining in the region $\{(a,b)\in \mathbb{R}^2:b\leq s\}$
and  it crosses the half-line $\{(r,b)\in \mathbb{R}^2:b\leq s\}$ of the dual lattice.

Analogously, we find a vertex $u_{i_2}$ where $i_2>i_0+1$ such that the path
$\lfloor \tilde{\gamma}^{\ast}\rfloor^{-1}(u_1,u_{i_2})$ leaves the vertex $u_1$ through the edge
$f_{\|x\|+m}^{\ast}\neq e_k^{\ast}$, does not crosses the ray $\{(a,0)\in \mathbb{R}^2:a\leq k\}$ of the
original lattice, always remains in the region $\{(a,b)\in \mathbb{R}^2:a\leq r\}$ and crosses
the ray $\{(a,s)\in \mathbb{R}^2:a\leq r\}$ of the dual lattice.
%
%

Then there must be a point of intersection of the paths $\lceil \tilde{\gamma}^{\ast}\rceil$ and $\lfloor \tilde{\gamma}^{\ast}\rfloor^{-1}$ distinct from $u_1$ and $u_{i_0+1}$, which is a contradiction. Then, $\gamma^{\ast}$ is a circuit in the dual having $\{0,x\}$ in its interior.

Therefore, every $\tilde{\gamma} \in
\Gamma^{\|x\|+m}_{\{0,x\}}(e_k^{\ast})$ is associated with a word
$W(\tilde{\gamma}^{\ast})$ containing a subword $W$ such that the
circuit $\gamma^{\ast}$ given by the subword $W$ from $u$ is the
dual of a minimal contour $\gamma$. Since the map
$W:\Gamma^{\|x\|+m}_{\{0,x\}}(e_k^{\ast})\rightarrow\{L,R,U,D\}^{\|x\|+m}$
is injective, an upper bound for the number of such words may be
obtained as follows: choose $\|x\|$ among $\|x\|+m$ entries to
allocate the subword associated with a minimal contour, then we
have $\beta_x$ possibilities to such subwords and in the remaining
$m$ entries we have at most 3 possibilities, because the circuit is
self-avoiding. Therefore we have
\begin{eqnarray}
|\Gamma^{\|x\|+m}_{\{0,x\}}(e_k^{\ast})|\leq 3^m \cdot \binom{\|x\|+m}{m}\beta_{x} \nonumber
\end{eqnarray}
for $k=-m+1,...,-1,0$
\end{proof}

{\bf Acknowledgments.} B.N.B. de Lima and R. Sanchis are partially supported by Conselho Nacional de Desenvolvimento Cient\'\i fico e Tecnol\'ogico (CNPq) and Funda\c c\~ao de Amparo \`a Pesquisa do Estado de Minas Gerais (Programa Pesquisador Mineiro). We thank the authors of the referee reports.

\end{document}